\documentclass[12pt]{amsart}
\usepackage{tikz,array, verbatim}
\usepackage{amsfonts, amsmath, latexsym, epsfig, caption}
\usepackage{amssymb, color}

\usepackage{epsf}
\usepackage{array}
\usepackage{ragged2e}
\usepackage{hyperref}
\usepackage[margin=1in]{geometry}
\usepackage{longtable}

\title[Bounds for the Regularity Radius of Delone Sets]{Bounds for the Regularity Radius of Delone Sets
}

\begin{document}

\author[N. Dolbilin]{Nikolay Dolbilin}
\address{Nikolay Dolbilin, Steklov Mathematical Institute, 8 Gubkina str., Moscow, 119991, Russia}
\email{dolbilin@mi-ras.ru}

\author[A. Garber]{Alexey Garber}
\address{Alexey Garber, School of Mathematical \& Statistical Sciences, The University of Texas Rio Grande Valley, 1 West University Blvd, Brownsville, TX, 78520, USA}
\email{alexey.garber@utrgv.edu}

\author[E. Schulte]{Egon Schulte}
\address{Egon Schulte, Department of Mathematics, Northeastern University, Boston, MA, 02115, USA}
\email{e.schulte@northeastern.edu}

\author[M.Senechal]{Marjorie Senechal}
\address{Marjorie Senechal, Department of Mathematics and Statistics, Smith College, Northampton, MA 01063, USA}
\email{senechal@smith.edu}

\newcommand{\R}{\ensuremath{\mathbb{R}}}

\newtheorem{theorem}{Theorem}[section]
\newtheorem{proposition}[theorem]{Proposition}
\newtheorem{corollary}[theorem]{Corollary}
\newtheorem{lemma}[theorem]{Lemma}
\newtheorem{problem}[theorem]{Problem}
\newtheorem{conjecture}[theorem]{Conjecture}
\newtheorem{question}{Question}
\newtheorem{claim}{Claim}
\newtheorem{observation}{Observation}

\theoremstyle{definition}
\newtheorem{definition}[theorem]{Definition}

\newtheorem{remark}[theorem]{Remark}

\begin{abstract}
Delone sets are discrete point sets $X$ in $\mathbb{R}^d$ characterized by parameters $(r,R)$, where (usually) $2r$ is the smallest inter-point distance of $X$, and $R$ is the radius of a largest ``empty ball" that can be inserted into the interstices of $X$. The regularity radius $\hat{\rho}_d$ is defined as the smallest positive number $\rho$ such that each Delone set with congruent clusters of radius $\rho$ is a regular system, that is, a point orbit under a crystallographic group. We discuss two conjectures on the growth behavior of the regularity radius. Our ``Weak Conjecture" states that $\hat{\rho}_{d}={{\rm O}(d^2\log d)}R$ as $d\rightarrow\infty$, independent of~$r$. This is verified in the paper for two important subfamilies of Delone sets: those with full-dimensional clusters of radius $2r$ and those with full-dimensional sets of $d$-reachable points. We also offer support for the plausibility of a ``Strong Conjecture", stating that $\hat{\rho}_{d}={{\rm O}(d\log d)}R$ as $d\rightarrow\infty$, independent of $r$.
\end{abstract}

\maketitle

\section{Introduction}
\label{intro}

In this paper we continue the mathematical study of crystal structure from a local viewpoint, to understand the conditions under which a discrete point set in  $\mathbb{R}^d$ is an orbit (or the union of a finite number of orbits) of a crystallographic group.  

Our main focus is on ``self-assembly'' properties of point sets in $\mathbb R^d$ as usual model for atomic structure of a solid structure. That is we study how local conditions for clusters of point sets imply periodicity of the global points sets. A good example of a similar self-assembly for tilings can be seen in the recently announced breakthrough solution to the ``Einstein problem" \cite{SMKG-S2023a,SMKG-S2023b} where one tile can self-assemble itself into a tiling of $\R^2$ but only in an aperiodic way.

One of the ways to model atomic structures is by using point sets that avoid small interpoint distances and large voids. This is usually done by assuming that the set $X \subset \mathbb{R}^d$ under investigation is a Delone set. Thus we assume that for the point set $X \subset \mathbb{R}^d$ there are positive real numbers $r$ and $R$ with $r\leq R$ such that
\begin{itemize}
\item an open ball $B_{\mathbf y}^o(r)$ of radius $r$  centered at any ${\mathbf y} \in \mathbb{R}^d$ contains \emph{at most} one point of $X$; and
\item a closed ball $B_{\mathbf y}(R)$ of radius $R$ centered at ${\mathbf y}$ contains \emph{at least} one point of $X$.   
\end{itemize}
In this paper we assume that the parameters $r$ and $R$ are known, unless otherwise stated. 
 
Delone sets are a modern version of the point sets B.~N.~Delone introduced in the early 1930s; he called them $(r,R)$ systems; see, e.g., {\cite{DAP}}. Here we discuss relations between the local and global (group-theoretic) viewpoints for the special case of regular (or multiregular) point sets.

A Delone set  $X$ is a \emph{regular system} if its symmetry group {$S(X)$} is point-transitive, i.e.., if for any points ${\mathbf x}$ and ${\mathbf x'}$ of $X$ there is an isometry $g$ such that $g({\mathbf x}) = {\mathbf x'}$ and $g(X) = X$. For each radius $\rho > 0$, the isometry $g$ also maps the $\rho$-cluster $C_{\mathbf x}(\rho):= X \cap B_{\mathbf x}(\rho)$ of ${\mathbf x}$ onto the $\rho$-cluster $C_{\mathbf x'}(\rho)$ of ${\mathbf x'}$. 

The number $N(\rho)$ of non-congruent (non-equivalent) cluster classes is called the cluster-counting function of $X$; it is non-decreasing. {If $X$ is regular, then clearly $N(\rho)=1$ for all $\rho > 0$.}  We also study the group $S_{\mathbf x}(\rho)$ of isometries that stabilize both ${\mathbf x}$ and the cluster $C_{\mathbf x}(\rho)$.    {These cluster groups $S_{\mathbf x}(\rho)$ are non-increasing.} Note that $N(\rho)$ and $S_{\mathbf x}(\rho)$ are interrelated: for example, $X$ is regular if $N(2R)=1$ and $C_{\mathbf x}(2R)$ is centro-symmetric \cite{DM2015}.

A criterion for regularity, due to Delone, Dolbilin, Shtogrin, and Galiulin \cite{local}, is fundamental for the local theory; we state and discuss it in the next section. 

We define the \emph{regularity radius} ${\hat\rho_d}=\hat\rho_d(r,R)$ to be the smallest value of the radius $\rho$ of the clusters such that every Delone set with $N(\rho) = 1$ is regular.  The main problem of the local theory is to find an upper bound for $\hat\rho_d$.  The order of the $2R$-cluster group $S_{\mathbf x}(2R)$ and the behavior of the groups $S_{\mathbf x}(\rho)$ play key roles.  

For $d = 2,3$ it is known that $\hat\rho_2=4R$ \cite{Dol2015,Dol2018a} and $6R \leq \hat\rho_3 \leq 10R$ \cite{LowerBound,DGLS}. Moreover, for  $d = 3$,  if $N(2R) = 1$, then the order of any axis in $S_{\mathbf x}(2R)$ does not exceed~6 {\cite{Sto2010}}. This is a ``local version'' of the so-called ``crystallographic restriction,'' which states that the orders of rotation in a lattice in $\mathbb{R}^2$ and $\mathbb{R}^3$ can only be $2, 3, 4$ or $6$.

In this paper we offer two conjectures about the growth rate of $\hat{\rho}_{d}$. Our Weak Conjecture claims that $\hat{\rho}_{d}={{\rm O}(d^2\log d)}R$ as $d\rightarrow\infty$, independent of $r$; and our Strong Conjecture proposes the even lower bound $\hat{\rho}_{d}={{\rm O}(d\log d)}R$ as $d\rightarrow\infty$, again independent of $r$. We verify the Weak Conjecture in important special cases that naturally originate from perceived atomic structure of crystals, and also provide some plausibility for the Strong Conjecture.

The paper is organized as follows. In Section \ref{basics} we introduce basic notions related to Delone sets and regularity radius and recall main results that describe how local structure of Delone sets implies global symmetry.

In Section \ref{clustconj} we consider the family of $d$-dimensional Delone sets with congruent $2R$-clusters that satisfy the additional condition that the boundaries of the smaller $2r$-clusters are full-dimensional. For this family, we prove that the regularity radius grows at most as fast as ${{\rm O}(d^2)}R$ as $d\rightarrow \infty$, see Theorem \ref{thm:towerboundM1}. We all also formulate our main Conjectures \ref{conjw} and \ref{conjs} on the growth rate of the regularity radius for general $2R$-isometric Delone sets in this section.

In Section \ref{conjreach} we extend our study of the regularity radius to the family of $t$-reachable sets, see also \cite{BouDo,Dol76}. In Theorems \ref{thm:towerboundMd} and \ref{thm:a} we prove an upper bound ${{\rm O}(d^2\log d)}R$ for the family of $t$-reachable Delone sets provided $t\leq ar$ for some fixed constant $a$. These results confirm Conjecture~\ref{conjw} for the family of Delone sets under consideration. 

We would like to point out that the family of Delone sets of Theorem~\ref{thm:a} has a reasonable description as the family describing atomic structures of solid crystals, as the minimal distance between atoms in such crystals together with the types of atoms used may dictate the size of voids provided the structure is kept solid by particle interactions.

In the concluding Section \ref{sec:open} we formulate several open problems that provide further insight in how local structures and their symmetries give rise to global structures within Delone sets.

\section{Basic notions}
\label{basics}

Throughout, $\log$ will denote the binary logarithm. In statements concerning asymptotics this could be replaced by the natural logarithm. 

Points, or vectors, in $\mathbb{R}^d$ are denoted by bold-faced letters. If ${\mathbf x}\in\mathbb{R}^d$ and $\rho\geq 0$, we let $B_{\mathbf x}(\rho)$ and 
$B_{\mathbf x}^o(\rho)$ denote the closed respectively open $d$-ball in $\mathbb{R}^d$ of radius $\rho$ centered at ${\mathbf x}$. For $M\subseteq\mathbb{R}^d$ we let ${\rm aff}\,M$ denote the affine hull of $M$. The dimension of ${\rm aff}\,M$ is called the {\em rank\/} of $M$ and is denoted ${\rm rk}\,M$.

Delone sets were already defined in the Introduction. We adopt the convention that in designating the parameters $(r,R)$ of a Delone set $X$ we choose the largest possible value of $r$ and the smallest possible value of $R$ that satisfy the two defining conditions. Thus, the smallest inter-point distance (if it exists) is $2r$ and the radius of the largest ``empty ball'' of $X$ (if it exists) is $R$. (Recall that a closed ball in $\mathbb{R}^d$ is called an {\em empty ball\/} of $X$ if no point of $X$ lies in its interior.)  Thus, the values of $r$ and $R$ are the ``exact" parameters of~$X$.

We later make assumptions on $X$ which guarantee that the smallest inter-point distance is attained and that there are largest empty balls.         

Let $X$ be a Delone set and $\rho\geq 0$. For a point ${\mathbf x}\in X$, we call 
\[C_{\mathbf x}(\rho):=X\cap B_{\mathbf x}(\rho)\] 
the {\em cluster of radius $\rho$} of $X$ with \emph{center} ${\mathbf x}$, or simply the {\em $\rho$-cluster of $X$ at ${\mathbf x}$}. The \emph{$\rho$-cluster group\/} $S_{\mathbf x}(\rho)$ at $\mathbf x$, is the stabilizer of $\mathbf x$ in the full symmetry group of $C_{\mathbf x}(\rho)$. Thus $S_{\mathbf x}(\rho)$ consists of all isometries of $\mathbb{R}^d$ which fix ${\mathbf x}$ and map $C_{\mathbf x}(\rho)$ to itself. Note that $S_{\mathbf x}(\rho)$ consists of $d$-dimensional isometries, even though the underlying cluster itself may lie in a subspace of dimension smaller than $d$ (and then also have a symmetry group in that subspace). 

Clusters of Delone sets are finite point sets. If the {rank of a cluster} is at most $d-2$, the corresponding cluster group must necessarily be infinite. In particular, if $\rho<2r$, then $C_{\mathbf x}(\rho)=\{\mathbf x\}$ and $S_{\mathbf x}(\rho)$ is isomorphic to the full orthogonal group of~$\mathbb{R}^d$. On the other hand, if the {rank of a cluster} is at least $d-1$, then since clusters are finite, the corresponding cluster group is finite as well. For $\rho\geq 2R$, all $\rho$-clusters of $X$ are $d$-dimensional, meaning that the affine hull is $\mathbb{R}^d$, and thus have finite cluster groups~\cite{local}. Note that the cluster groups of $X$ at a point $\mathbf{x}$ are non-increasing as a function of $\rho$, that is, $S_{\mathbf x}(\rho)\supseteq S_{\mathbf x}(\rho')$ whenever $\rho\leq\rho'$. In particular, $S_{\mathbf x}(\rho)$ is a subgroup of the finite group $S_{\mathbf x}(2R)$ whenever  $\rho\geq 2R$. Hence there is a natural interest in bounding the size of $S_{\mathbf x}(2R)$.

Two $\rho$-clusters $C_{\mathbf x}(\rho)$ and $C_{{\mathbf x}'}(\rho)$ at points ${\mathbf x},{\mathbf x}'$ of $X$ are said to be \emph{equivalent} if there exists an isometry of $\mathbb{R}^d$ that maps ${\mathbf x}$ to ${\mathbf x}'$ and $C_{\mathbf x}(\rho)$ to $C_{{\mathbf x}'}(\rho)$. As the isometry must map centers to centers,  equivalence of clusters is stronger than mere congruence of clusters.  Note that the cluster groups of any two equivalent $\rho$-clusters are conjugate subgroups of $\rm{Iso}(d)$, the group of all Euclidean isometries of $\mathbb{R}^d$. 

A Delone set $X$ in $\mathbb{R}^d$ is said to be {$\rho$-{\em {isometric}\/}} if any two $\rho$-clusters of $X$ are equivalent. We are particularly interested in $2R$-{isometric} Delone sets. 

A Delone set $X$ in $\mathbb{R}^d$ is called a {\em regular system\/} if its symmetry group $S(X)$ acts transitively on $X$. Thus a regular system coincides with the orbit of any one of its points under the  symmetry group. A regular system is $\rho$-{isometric} for every $\rho\geq 0$.

Cluster groups occur naturally in the following {\em Local Theorem for Delone Sets\/}, which provides a characterization of the regularity of a Delone set. 

\begin{theorem}
[Local Regularity Criterion, \cite{local}]\label{thm:local} 
A Delone set $X$  in $\mathbb{R}^d$ is a regular system if and only if $X$ satisfies the following two conditions for some $\rho_0>0$ and some ${\mathbf x}_0\in X$:
\begin{itemize}
\item $X$ is $(\rho_0 + 2R)$-{isometric} (that is, any two clusters of $X$ of radius $\rho_0 + 2R$ are equivalent).
\item $S_{{\mathbf x}_{0}}(\rho_0) = S_{{\mathbf x}_{0}}(\rho_0 + 2R)$.
\end{itemize}
Moreover, in this situation, $S_{\mathbf x}(\rho_0) = S_{\mathbf x}(\rho_0 + 2R)$ for all ${\mathbf x}\in X$; the cluster groups stabilize at radius $\rho_0$, that is, $S_{\mathbf x}(\rho) = S_{\mathbf x}(\rho_0)$ for all ${\mathbf x}\in X$ and all $\rho\geq\rho_0$; and the entire regular system $X$ is uniquely determined by the cluster $C_{{\mathbf x}_{0}}(\rho_0 + 2R)$.
\end{theorem}

Theorem~\ref{thm:local} provides an affirmative answer to the fundamental question asking whether or not the regularity of Delone sets in $\mathbb{R}^d$ can be recognized on clusters of bounded radius. The smallest such radius, over all Delone sets with parameters $(r,R)$, is called the {\em regularity radius\/} and is denoted~$\hat{\rho}_d$. Thus $\hat{\rho}_d$ is the smallest positive number $\rho$ with the property that each $\rho$-{isometric} Delone set $X$ is a regular system. From Dolbilin, Lagarias \& Senechal~\cite{DLS1998} it is known that 
\begin{equation}
\label{dls98}
\hat{\rho_d} \leq 2R(d^2+1) \log (2R/r+2).
\end{equation}
Our Conjectures~\ref{conjw} and \ref{conjs} below propose an upper bound for $\hat{\rho}_d$ which is independent of~$r$.

The next theorem is a reformulation of \cite[Prop. 2.1]{Dol2018} (see also \cite{DGLS}) and  provides an important tool for bounding the regularity radius $\hat{\rho}_d$ in terms of the dimension~$d$ and the radius  $R$ of the largest empty ball. It gives a necessary condition for the regularity of a Delone set $X$ in terms of the total number of prime factors of the order of its $2R$-cluster groups. 

For a finite group $G$, let $\Omega(G)$ denote the number of prime factors of $|G|$ (counted with multiplicity). Then, 
\begin{equation}
\label{omg}
\Omega(G) +1\,\leq\,\log (2|G|).
\end{equation}
If $X$ is a Delone set, $\mathbf{x}\in X$, and $\rho>0$, we define 
$\Omega_{\mathbf x}(\rho):= \Omega(S_{\mathbf{x}}(\rho))$, which is the number of prime factors of the order of the $\rho$-cluster group $S_{\mathbf{x}}(\rho)$. This number will not depend on $\mathbf{x}$ if $X$ is $\rho$-{isometric}. In our applications, $\rho=2r$ or a constant multiple of $2r$, or $\rho=2R$. 

\begin{theorem}
[Tower Bound]\label{thm:tower}
Let $X$ be a $2R$-{isometric} Delone set in $\mathbb{R}^d$, and let $\mathbf{x}$ be any point of $X$. If $X$ is also  $2(\Omega_{\mathbf{x}}(2R) +2)R$-{isometric}, then $X$ is a regular system.
\end{theorem}

The Tower Bound theorem shows that it is very desirable to establish good upper bounds for the size of the cluster groups in a Delone set, and thus for the total number of prime factors. In fact, the theorem shows that 
\begin{equation}
\label{rhod}
\\ \hat{\rho}_d\leq 2(\Omega_{\mathbf{x}}(2R) +2)R,
\end{equation}
for the family of $2R$-{isometric} Delone sets $X$ that have a bounded number of prime factors in the order of the corresponding cluster group. 

In this paper, we study how the regularity radius behaves on specific families $\mathcal{X}$ of Delone sets within the collection of all Delone set with parameters $(r,R)$. This will allow us to motivate our conjecture on an upper bound for $\hat{\rho}_d$.

For a family $\mathcal{X}$ of Delone sets, all with the same parameters $(r,R)$, the $\mathcal{X}$-{\em regularity radius\/} $\hat{\rho}_d(\mathcal{X})$ is the smallest positive number $\rho$ with the property that each $\rho$-{isometric} Delone set $X$ from $\mathcal{X}$ is a regular system. When $\mathcal{X}$ is the set of all Delone sets with parameters $(r,R)$, this is just the usual regularity radius $\hat{\rho}_d$. From the definition it is clear that $\hat{\rho}_d(\mathcal{X})\leq \hat{\rho}_d$ for each family $\mathcal{X}$. 

\section{Upper bounds for cluster groups and a conjecture for $\hat{\rho}_d$}
\label{clustconj}

In a Delone set $X$, 
all $2R$-clusters $C_{\bf x}(2R)$, ${\bf x}\in X$, are full-dimensional and have finite cluster groups $S_{\bf x}(2R)$. If $X$ is $2R$-{isometric}, then any two $2R$-cluster groups are isomorphic (in fact, conjugate in $\rm{Iso}(d)$). We are interested in finding upper bounds for the order of these groups. 

Let $X$ be a $2R$-{isometric} Delone set.  
For a point ${\mathbf x}\in X$, we let $M_{\mathbf x}(1)$ denote the subset of $X$ consisting of ${\mathbf x}$ and the points at minimal distance from $\mathbf{x}$, that is, 
\[M_{\mathbf x}(1) := C_{\mathbf x}(2r).\]
Then $M_{\mathbf x}(1)\subseteq C_{\mathbf x}(2R)$, and $M_{\mathbf x}(1)$ is invariant under $S_{\mathbf x}(2r)$. It follows that the affine hull of $M_{\mathbf x}(1)$ is also invariant under $S_{\mathbf x}(2r)$. But in general this subspace may not have full dimension $d$.

Recall that the {\em kissing number\/} $\tau_d$ for balls (spheres) is the maximum number of mutually non-overlapping congruent $d$-balls in $\mathbb{R}^d$ that can touch a given $d$-ball of the same size. 

The kissing number provides an easy upper bound for the cardinality of $M_{\mathbf x}(1)$, namely 
\[ |M_{\mathbf x}(1)|\leq 1+ \tau_d.\]
In fact, since any two points of $X$ are at least $2r$ apart, the balls of radius $r$ centered at points in 
$M_{\mathbf x}(1)\setminus \{\mathbf x\}$ all touch the ball of radius $r$ centered at ${\mathbf x}$ and are mutually non-overlapping. Hence there can be at most $\tau_d$ such balls. The currently best known upper bound for $\tau_d$ is 
\[\tau_d \leq 2^{0.4011d(1+\,\rm{o}(1))}\;(= 1.3205^{d(1+\,\rm{o}(1))}) \,\mbox{ as }d\rightarrow\infty\]
(see \cite{Bez2010,BDM,KaLe}).

\begin{proposition}
\label{groupM1}
Let $X$ be a Delone set 
in $\mathbb{R}^d$, let $\mathbf{x}\in X$, and suppose $M_{\mathbf x}(1)$ is full-dimensional. Then, 
\[ |S_{\mathbf{x}}(2R)|\leq |S_{\mathbf{x}}(2r)|\,
\leq
\tau_{d}\cdot \tau_{d-1}\cdot\ldots\cdot\tau_{1}
\,\leq\,\tau_d^{\,d}
,\]
and 
\[ \Omega_{\mathbf x}(2R)\leq \Omega_{\mathbf x}(2r)\leq d\log \tau_d
.\]
\end{proposition}

\begin{proof}
Without loss of generality we may assume that $\mathbf{x}$ is the origin. Then $S_\mathbf{x}(2r)$ is a finite subgroup of the full orthogonal group of $\mathbb{R}^d$ acting faithfully as a permutation group on the full-dimensional set $M_{\mathbf x}(1)$. Now choose a vector space basis of $\mathbb{R}^d$ from the elements of $M_{\mathbf x}(1)\setminus \{\mathbf x\}$. This is possible since $M_{\mathbf x}(1)$ is full-dimensional. Clearly, each element of $S_\mathbf{x}(2r)$ is uniquely determined by its effect on this basis. But $M_{\mathbf x}(1)\setminus \{\mathbf x\}$ contains at most $\tau_d$ elements, so there are at most $\tau_d$ choices for the image point of the first basis vector. 

The image of the second point is now restricted to a sphere of codimension 1 within the sphere of radius $2r$ centered at $\mathbf x$ because the distance between the images of the first and second basis vectors is fixed. This gives at most $\tau_{d-1}$ choices for the image point of the second basis vector once the image of the first is chosen. Using similar arguments, we have at most $\tau_{d-2}$ choices for the image point of the third basis vector once the images of the first and second are chosen, and so on. Hence, 
\[|S_\mathbf{x}(2r)| \leq \tau_{d}\cdot \tau_{d-1}\cdot\ldots\cdot\tau_{1}\leq \tau_d^{\,d}.\]
As $S_{\mathbf{x}}(2R)$ is a subgroup of $S_{\mathbf{x}}(2r)$, then this proves the first part. This subgroup relationship also establishes  the first inequality of the second part, while the second inequality uses (\ref{omg}) and follows directly from 
\[ \Omega_{\mathbf x}(2r) +1 \leq \log (2|S_{\mathbf x}(2r)|) \leq \log {2\tau_d^d} 
=1 + d\log\tau_d
.\]
\end{proof}

\begin{remark}
\label{remone}
Asymptotically, as $d\rightarrow\infty$, the smaller upper bound 
$\tau_{d}\cdot \tau_{d-1}\cdot\ldots\cdot\tau_{1}$ for $|S_{\mathbf{x}}(2R)|$ in Proposition~\ref{groupM1} is significantly better than the  larger upper bound $\tau_d^{\,d}$. However, after taking the logarithm and using the general upper bound $\tau_d \leq 2^{0.4011d(1+\,\rm{o}(1))}$, the resulting terms differ only by a constant factor and thus we use the simpler bound of $\tau_d^d$ instead.
\end{remark}

The kissing number argument underlying Proposition~\ref{groupM1} allows us to analyze the behavior of the regularity radius on the important family of $2R$-isometric Delone sets for which the points at minimal distance from a given point form a full-dimensional set.

\begin{theorem}
\label{thm:towerboundM1}
Suppose $\mathcal{X}$ is the family of $2R$-{isometric} Delone sets $X$ in $\mathbb{R}^d$ with the property that  $M_{\mathbf x}(1)$ is full-dimensional for some (and thus all) $\mathbf{x}\in X$. Then
\[ \hat{\rho}_d(\mathcal{X})\,\leq 2(d\log \tau_d +2)R \leq
\, 4+0.8022{d^2 (1+\,{\rm o}(1))}R \,=
\, {{\rm O}(d^2)}R \,\mbox{ as }  d\rightarrow\infty.\]
The ${\rm O}(d^2)$-term and the ${\rm o}(1)$-term only dependent on the dimension $d$, not the parameters of $X$.
\end{theorem}

\begin{proof}
This follows from Theorem~\ref{thm:tower} and Proposition~\ref{groupM1}. So let $X\in\mathcal{X}$; in other words, $X$ is a $2R$-{isometric} Delone set in $\mathbb{R}^d$ with the property that $M_{\mathbf x}(1)$ is full-dimensional for each $\mathbf{x}\in X$. By Proposition~\ref{groupM1}, 
\[\Omega_{\mathbf{x}}(2R)\leq d\log \tau_d\]
for each $\mathbf{x}\in X$, and therefore, using the upper bound for $\tau_d$ and Theorem~\ref{thm:tower}, we get 
\[\hat{\rho}_d(\mathcal{X})\leq 2(\Omega_{\mathbf{x}}(2R) +2)R \leq 2
(d\log \tau_d +2)R \leq 2(2+0.4011{d^2(1+\,\rm{o}(1)))}R.\]
This is the desired inequality.
\end{proof}
\smallskip

Based on the results above (and the discussion in Section~\ref{sec:open}) we offer two conjectures for the growth behavior of the regularity radius. The following {\it Weak Conjecture\/} is supported by both Theorem~\ref{thm:towerboundM1} above and Theorem~\ref{thm:towerboundMd} below. It states that the usual regularity radius $\hat{\rho}_d$ grows at most linearly in the radius of the largest empty ball $R$, and like $d^2\log d$ in the dimension $d$, independent of~$r$. In other words, $\hat{\rho}_{d}/R$ grows at most like $d^2\log d$ in the dimension $d$, independent of $r$.  

\begin{conjecture}{\rm (Weak Conjecture)}
  \label{conjw}
There exists a constant $c$, not depending on $d$ and $R$, such that $\hat{\rho}_d \leq c (d^2\log d)R$ for each $d\geq 1$.
\end{conjecture}

Note that Theorem~\ref{thm:towerboundM1} actually gives a slightly better bound than Conjecture~\ref{conjw} for that specific family $\mathcal{X}$, namely ${{\rm O}(d^2)}R$. Together with the lower bound of $2dR$ for $\hat{\rho}_d$ established in~\cite{LowerBound}, settling Conjecture~\ref{conjw} would narrow down the possible range for $\hat{\rho}_d$ to 
\[ 2dR\, \leq \,\hat{\rho}_d\, \leq\, c(d^2\log d)R \]
and therefore considerably close the gap between the best known lower and upper bounds for~$\hat{\rho}_d$.

Motivated by the linear lower bound of~\cite{LowerBound} and further plausibility considerations in Section~\ref{sec:open}, we are also proposing the following {\it Strong Conjecture\/} for the growth rate of~$\hat{\rho}_d$.

\begin{conjecture}{\rm (Strong Conjecture)}
\label{conjs}
There exists a constant $c$, not depending on $d$ and $R$, such that $\hat{\rho}_d \leq c (d\log d)R$ for each $d\geq 1$.
\end{conjecture}

\section{The conjecture and $2r$-reachability}
\label{conjreach}

There is further evidence in support of the Conjecture~\ref{conjw}, as we explain in this section. We begin by introducing the concept of $t$-reachability.

Let $X$ be a Delone set, and let $t$ with $2r\leq t$ be fixed. For $\mathbf{x}\in X$ and $k\geq 0$, let $M_{\mathbf{x}}(k)$ denote the set of all points $\mathbf{y}\in X$ for which there exists a finite chain of points 
\[\mathbf{x}=:\mathbf{x}_0,\mathbf{x}_1,\ldots,\mathbf{x}_{m-1},\mathbf{x}_{m}:=\mathbf{y}\]
in $X$ such that $0\leq m\leq k$ and $||\mathbf{x}_{i+1}-\mathbf{x}_i||\leq t$ for all $i=0,\ldots,m-1$. (We are suppressing the dependence on $t$ in our notation, as later the value of $t$ will be clear from the context.) Hence $M_{\mathbf{x}}(k)$ consists of all points $\mathbf{y}$ of $X$ that are {\em {$t$}-reachable\/} from $\mathbf{x}$ in at most $k$ steps of size at most $t$, and in particular contains $\mathbf{x}$ itself. In particular, $M_{\mathbf{x}}(0)=\{x\}$. We also let
\[ M_{\mathbf{x}}:=\bigcup_{k\geq 0} M_{\mathbf{x}}(k) \]
and call $M_{\mathbf{x}}$ the {\it set of $t$-reachable points} for $\mathbf{x}$. Note that $M_{\mathbf{x}}=M_{\mathbf{y}}$ if and only if ${\mathbf{y}}\in M_{\mathbf{x}}$. Clearly, $M_{\mathbf{x}}(k)\subseteq C_{\mathbf x}(kt)$ for all $k\geq 0$.

Our discussion of $t$-reachable points is inspired by the concept of $t$-bonded sets introduced in Dolbilin~\cite{Dol76} and developed in Bouniaev \& Dolbilin~\cite{BouDo}. A Delone set $X$ is called {\it $t$-bonded} for a parameter $t>0$ if every two points $\mathbf x,\mathbf y\in X$ can be connected by a finite chain 
$\mathbf x=\mathbf x_0,\ldots,\mathbf x_k=\mathbf y$ in $X$ such that 
$||\mathbf x_{i+1}-\mathbf x_{i}||\leq t$ for every $i$. 

Notice that $X$ is $t$-bonded if and only if $X$ coincides with the set of $t$-reachable points $M_{\mathbf x}$ for each ${\mathbf x}\in X$. In this case $M_{\mathbf x}=X$ and in particular, $M_{\mathbf x}$ is full-dimensional. On the other hand, as the example of 
\[X=\mathbb{Z}^d\cup ((\frac12,\ldots,\frac12)+\mathbb{Z}^d), \; d>4, \] 
with $t=2r$ shows, there exists a Delone set $X$ with full-dimensional sets $M_{\mathbf x}$ for every $\mathbf x\in X$, which, nevertheless, is not $2r$-bonded. 

It follows from the definition of the parameters $(r,R)$ that a Delone set $X$ is $t$-bonded for every $t\geq 2R$, and thus $M_{\mathbf x}=X$ in this case.

We are particularly interested in the case when $X$ is $2R$--{isometric} and $t=2r$. In this case successive points in the chain connecting $\mathbf{x}$ to $\mathbf{y}$ are exactly at distance $2r$. Moreover, when $k=1$ our notation is consistent with earlier use of the notation; more precisely, $M_{\mathbf x}(1)$ is the subset of $X$ consisting of $\mathbf x$ and the set of all points of $X$ at minimum distance $2r$ from $\mathbf x$.
\medskip

Our next theorem confirms Conjecture~\ref{conjw} under the assumption that the set of $2r$-reachable points $M_{\mathbf{x}}$ is $d$-dimensional for each $\mathbf{x}\in X$.   

\begin{theorem}
\label{thm:towerboundMd}
Suppose $\mathcal{X}$ is the family of $2R$-{isometric} Delone sets $X$ with the property that the set of $2r$-reachable points $M_{\mathbf x}$ is full-dimensional for each $\mathbf{x}\in X$. Then
\[ \hat{\rho}_d(\mathcal{X}) \leq  {{\rm O}(d^2\log d)}R \,\mbox{ as } d\rightarrow\infty.\]
The ${\rm O}(d^2\log d)$-term only depends on the dimension $d$, not the parameters of $X$.
\end{theorem}

\begin{proof}
We split $\mathcal{X}$ into two subfamilies and deal with each subfamily separately. Let $\mathcal{X}'$ and $\mathcal{X}''$ denote the subfamilies of $\mathcal{X}$ consisting of all Delone sets $X$ in $\mathcal{X}$ with $2dr\leq 2R$ or $2dr>2R$, respectively. 

First suppose that $X$ lies in $\mathcal{X}'$. In this case we claim that already $M_{\mathbf x}(d)$ is full-dimensional for each ${\mathbf x}\in X$. Suppose that 
{${\rm rk}\, M_{\mathbf x}(d)\leq d-1$} for some $\mathbf x\in X$. Then there is $1<k\leq d$ such that 
$${\rm rk}\, M_{\mathbf x}(k)= {{\rm rk}\, M_{\mathbf x}(k-1)}=:m\leq d-1.$$ 
We use induction to show that then each $M_{\mathbf x}(l)$ with $l\geq k$, and thus $M_{\mathbf x}$ itself, must also have rank $m$. 

First, suppose $\mathbf y$ is a point in $M_{\mathbf x}(k+1)$. We claim that $\mathbf y\in {\rm aff}\, M_{\mathbf x}(k)$. Indeed, there is a point $\mathbf z\in M_{\mathbf x}(1)$ such that $\mathbf y\in M_{\mathbf z}(k)$. Then, since $2dr\leq 2R$ and $k\leq d$, the $2kr$-clusters of $\mathbf x$ and $\mathbf z$ are equivalent, and the $2(k-1)r$-clusters of $\mathbf x$ and $\mathbf z$ also are equivalent. In particular, $M_{\mathbf x}(k)$ and $M_{\mathbf z}(k)$ {have rank} $m$, and $M_{\mathbf x}(k-1)$ and $M_{\mathbf z}(k-1)$ also have rank~$m$. Since $M_{\mathbf z}(k-1)$ is a subset of both $M_{\mathbf x}(k)$ and $M_{\mathbf z}(k)$, the $m$-dimensional spans of $M_{\mathbf x}(k)$ and $M_{\mathbf z}(k)$ coincide and thus $\mathbf y\in {\rm aff}\, M_{\mathbf x}(k)$. 

Now, since $\mathbf y$ was an arbitrary point in $M_{\mathbf x}(k+1)$, these arguments show that then also {${\rm rk}\,M_{\mathbf x}(k+1) = m$}. This settles the case $l=k+1$.

Proceeding inductively, we can apply similar arguments to show that {${\rm rk}\,M_{\mathbf x}(l)=m$} for each $l\geq k$. Then it follows that {${\rm rk}\, M_{\mathbf x}=m<d$}, which is a contradiction to our assumption on $X$. Thus $M_{\mathbf x}(d)$ must be full-dimensional for each $x\in X$ if $X$ lies in $\mathcal{X}'$.

Since $X$ is $2R$-{isometric} and 
$M_{\mathbf x}(d)\subseteq C_{\mathbf x}(2dr)\subseteq C_{\mathbf x}(2R)$ 
for each ${\mathbf x}\in X$, any two sets $M_{\mathbf x}(d), M_{\mathbf x'}(d)$ with ${\mathbf x},{\mathbf x'}\in X$ are congruent and thus have the same cardinality denoted $\mu$. We claim that 
\[ \mu\leq (2d+1)^d .\]
In fact, since the smallest inter-point distance of $X$ is $2r$, it follows that for any ${\mathbf x}\in X$, the family of $d$-balls $B_{\mathbf y}(r)$, ${\mathbf y}\in M_{\mathbf x}(d)$, forms a packing of $d$-balls of radius $r$ all contained in $B_{\mathbf x}(2dr +r)$. If $\kappa_d$ denotes the volume of the $d$-ball of radius $1$, then the packing property shows that $\mu\, r^d\kappa_{d} \leq (2dr+r)^d\kappa_d$ and thus $\mu\leq (2d+1)^d$. 

A similar argument as at the beginning of the proof of Proposition~\ref{groupM1} then shows that  
\[  |S_\mathbf{x}(2R)|\leq |S_\mathbf{x}(2dr)|\leq \mu^{d}\leq (2d+1)^{d^2}.\]
(In this case there are at most $\mu$ choices for the image point of each basis vector.) 
Hence
\[\Omega_{\mathbf x}(2R) +1 \leq \log (2|S_{\mathbf x}(2R)|)\leq 1+d^{2}\log (2d+1)\]
and in analogy to (\ref{rhod}),
\[ \hat{\rho}_d(\mathcal{X}') \leq 2(\Omega_{\mathbf{x}}(2R) +2)R \leq (4+2d^{2}\log (2d+1))R 
={{\rm O}(d^2\log d)}R.\]

The case when $X$ lies in $\mathcal{X}''$ can be settled by appealing to the upper bound for $\hat{\rho_d}$ established in Dolbilin, Lagarias \& Senechal~\cite{DLS1998}, 
\[ \hat{\rho_d} \leq 2R(d^2+1) \log (2R/r+2).\]
Since now $2dr>2R$, this gives 
\[ \hat{\rho}_d(\mathcal{X}'')\leq 2R(d^2+1) \log (2d+2)
={{\rm O}(d^2\log d)}R.\]

Finally, combining the bounds for the two families $\mathcal{X}'$ and $\mathcal{X}''$ we obtain the desired upper bound for $\mathcal{X}$ itself, 
\[ \hat{\rho}_d(\mathcal{X}) 
= {\rm max}\big(\hat{\rho}_d(\mathcal{X}'),\hat{\rho}_d(\mathcal{X}'')\big)
={\rm O}(d^2\log d)R.\]
\end{proof}
\smallskip

Our next theorem is a generalization of Theorem~\ref{thm:towerboundMd}. It provides further evidence in support of Conjecture~\ref{conjw} and exploits $t$-reachability for other values of $t$. 

Let $a\geq 1$ be a fixed real number, and set $t:=2ar$. Write $M_{\mathbf{x}}^{a}(k)$ instead of $M_{\mathbf{x}}(k)$ for the set of all points of $X$ that can be reached from $\mathbf{x}$ in at most $k$ steps of size at most $2ar=t$, and write $M^a_{\mathbf{x}}$ instead of $M_{\mathbf{x}}$ for the set of $2ar$-reachable points for $\mathbf{x}$. When $a=1$ we are back in the case $t=2r$.

\begin{theorem}
\label{thm:a}
Suppose $\mathcal{X}_a$, $a\geq 1$, is the family of $2R$-{isometric} Delone sets $X$ with the property that $M^a_{\mathbf{x}}$ is full-dimensional for each $\mathbf{x}\in X$. Then
\[ \hat{\rho}_d(\mathcal{X}_a) =  {{\rm O}(d^2\log d)}R \,\mbox{ as } d\rightarrow\infty,\]
with an ${{\rm O}(d^2\log d)}$-term depending on both $d$ and $a$.
\end{theorem}
 
\begin{proof}
Similarly to the proof of Theorem \ref{thm:towerboundMd}, we split $\mathcal{X}_a$ into two subfamilies and deal with each subfamily separately. Let $\mathcal{X}'_a$ and $\mathcal{X}''_a$ denote the subfamilies of $\mathcal{X}$ consisting of all Delone sets $X$ in $\mathcal{X}$ with $2dar\leq 2R$ or $2dar>2R$, respectively.

If $X\in \mathcal{X}'_a$, then we can use the same arguments as in the proof of Theorem \ref{thm:towerboundMd} to show that already $M^a_{\mathbf x}(d)$ is full-dimensional for every $\mathbf x\in X$. Yet another appeal to the proof of Theorem \ref{thm:towerboundMd} then shows that the cardinality $\mu$ of $M^a_{\mathbf x}(d)$ satisfies the inequality $\mu\leq (2da+1)^d$. Hence
\[  |S_\mathbf{x}(2R)|\leq |S_\mathbf{x}(2dar)|\leq \mu^{d}\leq (2da+1)^{d^2}\]
and
\[\Omega_x(2R) +1 \leq \log (2|S_{x}(2R)|)\leq 1+d^{2}\log (2da+1).\]
Proceeding as in (\ref{rhod}) then establishes 
\[ \hat{\rho}_d(\mathcal{X}'_a) \leq 2(\Omega_{\mathbf{x}}(2R) +2)R \leq (4+2d^{2}\log (2da+1))R 
={{\rm O}(d^2\log d)}R,\]
but now the ${{\rm O}(d^2\log d)}$-term depends on $a$ as well.

The case when $X$ belongs to $\mathcal{X}''_a$ can be settled as in the previous theorem by using the bound from Dolbilin, Lagarias \& Senechal~\cite{DLS1998}. Combining the bound 
\[ \hat{\rho_d} \leq 2R(d^2+1) \log (2R/r+2)\] from \cite{DLS1998} with the inequality $2dar>2R$ we obtain 
\[ \hat{\rho}_d(\mathcal{X}''_a)\leq 2R(d^2+1) \log (2da+2)
={{\rm O}(d^2\log d)}R,\]
again with an ${{\rm O}(d^2\log d)}$-term that depends on $a$.

The statement of the theorem now follows from combining the two upper bounds for $\hat{\rho}_d(\mathcal{X}'_a)$ and $\hat{\rho}_d(\mathcal{X}''_a)$.
\end{proof}

Note that the growth behavior of the two upper bounds for $\hat{\rho}_d(\mathcal{X}'_a)$ and $\hat{\rho}_d(\mathcal{X}''_a)$ in the proof of Theorem~\ref{thm:a} is essentially controlled by 
\[ d^2 \log (da) = d^2 \log (d) + d^2 \log (a)\]
as the dominant term. This shows that a change in the reachability step size from $2r$ to $2ar$, $a\geq 1$, results (essentially) in the addition of a term $d^2\log (a)$ to the bound of Theorem~\ref{thm:towerboundMd} for $a=1$.
\medskip

\begin{remark}
\label{remtwo}
In Theorem~\ref{thm:a}, if $2ar\geq 2R$ then $M^a_{\mathbf{x}}=X$ and $\mathcal{X}_a$ coincides with the family of all Delone sets with parameters $(r,R)$. In particular this holds for $a=R/r$. In Theorem~\ref{thm:a} it therefore suffices to assume that $1\leq a \leq R/r$. For any $a< R/r$ there exists a Delone set $X$ where the sets of $2ar$-reachable points $M^a_{\mathbf{x}}$, $x\in X$, are proper subsets of $X$ of rank less than $d$. On the other hand, the extreme case $a=R/r$ is covered by Theorem~\ref{thm:a}, and in this case the bound of Theorem~\ref{thm:a} provides an explicit upper bound for the regularity radius for the whole family of Delone sets that depends on $R$ and $R/r$. 
\end{remark}

Concluding this section, we note that while our asymptotic results describe the behavior of the regularity radius in high dimensions rather than in ordinary 3-space (where we have better results~\cite{DGLS}), we are expecting that a vast majority of structured solid materials can be modelled using Delone sets with full-dimensional sets $M^a_{\mathbf x}$ for some $a\geq 1$; the reason is that such {\it short-range} interactions between atoms should more likely enforce a solid structure than the lack of such interactions.

\section{Open Problems}
\label{sec:open}

We conclude with several open problems. 

\subsection{Group Order Problem.}
This problem concerns the size of the $2R$-cluster groups $S_{\mathbf x}(2R)$ in a $2R$-isometric Delone set in $\mathbb{R}^d$. We let $h_d$ denote an upper bound for the order $|S_{\mathbf x}(2R)|$ of these groups, and note that our notation suppresses a possible dependence on the parameters of $X$. It is known that $h_2 = 12$ and $h_3 = 48$. For $d \geq 4$, there is a known upper bound depending on the parameters of $X$ that comes from the cardinality of the $2R$-clusters~\cite{DLS1998}. The problem is to find, for $d \geq 4$, an upper bound $h_d$ that does not depend on the ratio $R/r$. 

We conjecture that $h_d = 2^d\cdot d!$ for all sufficiently large $d$, independent of the parameters (see also~\cite{Fri97}). If true, this implies a better bound on the regularity radius than the ${{\rm O}(d^2\log d)}R$ bound of our Weak Conjecture~\ref{conjw}. In fact, if $h_d = 2^d\cdot d!$ for all sufficiently large $d$, then
 \[ \Omega_{\mathbf x}(2R) +1 \leq \log (2|S_{\mathbf x}(2R)|) 
\leq \log (2h_d) = 1+d+ \log (d!) = {\rm O}(d\log d)R,\]
independent of $r$, and therefore also 
$\hat{\rho}_{d}={\rm O}(d\log d)R$. This lends plausibility to our Strong Conjecture~\ref{conjs}.

The case $d=4$ deserves special attention. We are proposing to find a positive constant $c$, independent of the parameters $(r,R)$, such that 
\[\hat{\rho}_{4}\leq 2cR.\]  
In this case it would be helpful to establish a good upper bound for the order of the $2R$-cluster groups in a $2R$-isometric 4-dimensional Delone set. 

\subsection{Drop of Symmetries.} Our next question concerns the stabilization of the cluster group as the radius of the cluster grows.

Let $X$ be a {\em regular\/} Delone set, and let $\mathbf x \in X$. We define the {\it drop of $X$}, denoted $D(X)$ as the number of proper inclusions in the sequence of the subgroups
$$S_{\mathbf x}(2R)\supseteq S_{\mathbf x}(4R) \supseteq S_{\mathbf x}(6R) \supseteq \ldots$$
Note that, since $X$ is regular, the number of proper inclusions does not depend on the point~$\mathbf x$.

If $X$ is locally asymmetric, that is, if $S_{\mathbf x}(2R)$ is trivial, then clearly $D(X)=0$. However, if the symmetry group of the $2R$-cluster of $X$ is rich, then theoretically the drop could be as large as the number of prime factors in its order. 

On the other hand, the treatment of Delone sets with octahedral and tetrahedral symmetries in $\R^3$ in \cite{Dol2018} suggests that rich local symmetries may enforce regularity without requiring too many drops.


Let $D_d$ denote the maximal drop $D(X)$ among all regular $d$-dimensional Delone sets~$X$. We are asking to determine or bound $D_d$ as a function of $d$.

Note that regular Delone sets $X$ among the Engel sets of~\cite{LowerBound} give examples in dimension~$d$ with $D(X)=d-2$, and therefore $D_d$ is at least linear in $d$. We expect that a detailed study of the maximal drop value $D_d$ and of constructions of regular sets with maximal drop may lead to improved lower bounds for the regularity radius, much like the construction of Engel sets leads to the currently best known lower bound.

\medskip

\subsection{Non-regular yet orderly Delone sets.} Delone sets are more than a descriptive tool for discrete point sets: as this paper shows, they are also a tool for characterizing the relation between global regularity and local structure. On the other hand, finding local criteria for orderly but non-regular sets (with ``orderly'' suitably defined) is still unexplored. Delone sets on the line would be one good place to start (see, for example,
\cite{Sen1995} and \cite{Bal2014}).

\subsection*{Acknowledgment}
We would like to thank the American Institute of Mathematics (AIM) for hosting our SQuaRe research project on ``Delaunay Sets: Local Rules in Crystalline Structures'', as well as the 2016  AIM workshop on ``Soft Packings, Nested Clusters, and Condensed Matter''. The present paper resulted from discussions initiated at the second SQuaRe meeting in 2019. We greatly appreciated the opportunity to meet at AIM and are grateful to AIM for the hospitality. The work of N.~D.   was performed at the Steklov International Mathematical Center and supported by the Ministry of Science and Higher Education of the Russian Federation (agreement no. 075-15-2022-265). The work of A.~G. was partially supported by the Alexander von Humboldt Foundation. The work of E.~S. was supported by Simons Foundation award no.~420718.

\end{document}